\documentclass[11pt]{article}

\usepackage{amsmath,amsthm,amsfonts,amssymb}
\usepackage[margin=2.2cm]{geometry}
\usepackage{thmtools}
\usepackage{hyperref}
\usepackage{enumerate}
\usepackage{authblk}
\usepackage{orcidlink}

\newtheorem{thm}{Theorem}
\newtheorem{lem}{Lemma}
\newtheorem{prop}{Proposition}[thm]
\newtheorem{cor}{Corollary}

\theoremstyle{definition}
\newtheorem{defin}{Definition}[thm]
\newtheorem{obs}{Observation}
\newtheorem{exmp}{Example}

\newcommand{\nullity}{\operatorname{nullity}}
\newcommand{\sgn}{\operatorname{sgn}}
\newcommand{\indeg}{\operatorname{indeg}}
\newcommand{\outdeg}{\operatorname{outdeg}}
\usepackage{bbold}
\usepackage{tikz}
\usetikzlibrary{matrix,arrows.meta}
\usepackage{enumitem}
\usetikzlibrary{positioning}
\usetikzlibrary{calc}
\usetikzlibrary{decorations.markings}

\title{The number of Pfaffian orientations on punctured polygonally cellulated surfaces} 
\author[a]{Sajal Mukherjee~\orcidlink{0009-0004-6959-4334}\thanks{\texttt{sajal.mukherjee@tcgcrest.org}}}

\author[b]{Pritam Chandra Pramanik~\orcidlink{0009-0000-8332-3347}\thanks{\texttt{pritam.pramanik.80@tcgcrest.org}}}

\author[c]{Arundhati Rakshit~\orcidlink{0009-0002-8209-8584}\thanks{\texttt{arundhati.rakshit.124@tcgcrest.org}}}

\affil[a,b,c]{\small Institute for Advancing Intelligence (IAI), TCG CREST, Kolkata--700091, West Bengal, India}
\affil[a,c]{\small Mathematical \& Information Science, Academy of Scientific and Innovative Research (AcSIR), Ghaziabad--201002, India}

\date{}

\begin{document}
	\maketitle
	
	\begin{abstract}
		 In this paper, we introduce the notion of Pfaffian orientations on (punctured) polygonally cellulated orientable surfaces, and provide an expression for the number of such orientations. This generalizes the notion of Pfaffian orientations on planar graph, where a planar graph is seen as a punctured $2$-sphere, embedded in $\mathbb{R}^3$. So as a direct corollary of our main theorem, we derive the number of Pfaffian orientations on a planar graph.
	
	\end{abstract}
	
	\section{Introduction}
 The problem of enumerating perfect matchings in planar graphs has originated in the realms of statistical mechanics and chemistry, more particularly, in the problems of dimer models. 	The notion of \emph{Pfaffian orientations} plays a central role towards this end, in counting the number of perfect matchings in planar graphs. Kastelyn has proved the existence of Pfaffian orientations in any bipartite planar graph (see \cite{Kasteleyn1963}, \cite{Kasteleyn}). However, there has been no known result on the number of such Pfaffian orientations, as of now, to the best of our knowledge. 
 
 A \emph{Pfaffian orientation on a planar graph} (see \cite{Arora2023}) is defined as a choice of orientation on the edges of the planar graph such that each face of the graph has odd number of edges oriented in the clockwise direction with respect to that face. \autoref{pf} demonstrates a Pfaffian orientation on a planar graph. \autoref{pf} demonstrates a Pfaffian orientation on a planar graph.
 
  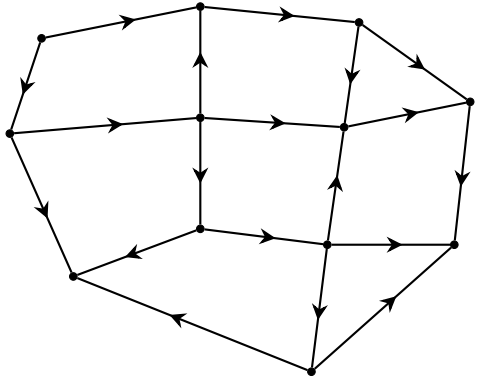
\begin{figure}
 	\centering
 	
 	\begin{tikzpicture}[
 		every node/.style={draw, circle, fill=black, inner sep=0.1pt, minimum size=1mm},
 		baseline, scale=2.1]
 		
 		\tikzset{
 			midarrow/.style={
 				postaction={
 					decorate,
 					decoration={
 						markings,
 						mark=at position 0.6 with {
 							\arrow{Stealth[length=6pt,width=6pt]}
 						}
 					}
 				}
 			}
 		}
 		
 		
 		\node (A) at (0,1.3) {};
 		\node (B) at (1.0,1.5) {};
 		\node (C) at (2.0,1.4) {};
 		
 		\node(D) at (-0.2,0.7) {};
 		\node (E) at (1.0,0.8) {};
 		\node (F) at (2.7,0.9) {};
 		
 		\node (H) at (1,0.1) {};
 		\node (I) at (1.8,0) {};
 		
 		\node (J) at (1.7,-0.8) {};
 		\node (K) at (0.2,-0.2) {};
 		\node (L) at (2.6,0) {};
 		\node (X) at (1.906,0.740) {};
 		
 		\draw[midarrow, thick] (A)--(D);
 		\draw[midarrow, thick] (A)--(B);
 		\draw[midarrow, thick] (B)--(C);
 		\draw[midarrow, thick] (D)--(K);
 		
 		\draw[midarrow, thick] (C)--(F);
 		\draw[midarrow, thick] (F)--(L);
 		
 		\draw[midarrow, thick] (J)--(L);
 		\draw[midarrow, thick] (J)--(K);
 		
 		
 		\draw[midarrow, thick] (D)--(E);
 		\draw[midarrow, thick] (E)--(X);
 		\draw[midarrow, thick] (X)--(F);
 		\draw[midarrow, thick] (I)--(X);
 		
 		\draw[midarrow, thick] (H)--(I);
 		\draw[midarrow, thick] (I)--(L);
 		
 		\draw[midarrow, thick] (E)--(B);
 		
 		\draw[midarrow, thick] (H)--(K);
 		\draw[midarrow, thick] (I)--(J);
 		
 		\draw[midarrow, thick] (E)--(H);
 		\draw[midarrow, thick] (C)--(X);
 		
 	\end{tikzpicture}
 	
 	\caption{A Pfaffian orientation on a planar graph. The arrows represent the orientations on the edges.}\label{pf}
 \end{figure}
In this paper, we introduce a more general notion of Pfaffian orientations on ``punctured" polygonally cellulated surfaces, embedded in $\mathbb{R}^3$, whose $1$-skeletons are not necessarily planar, and provide a count of all such Pfaffian orientations on a given ``punctured" polygonally cellulated orientable surface. Hence, we determine the count of Pfaffian orientations on a ``punctured" polygonally cellulated $2$-sphere.
 
 For any polygonally cellulated surface, the $2$-cells are termed as \emph{faces}, the $1$-cells are termed as \emph{edges} while the $0$-cells are termed as \emph{vertices}. From here onwards, we will refer to polygonally cellulated surfaces as simply cellulated surfaces for our convenience.
 
 Let $K$ be an orientable, connected, closed, cellulated surface of genus $g$, embedded in $\mathbb{R}^3$, where the orientation of each face is determined by the right hand rule along the outward normal vector of the face. \autoref{fig1} illustrates the orientation of one such face.
 
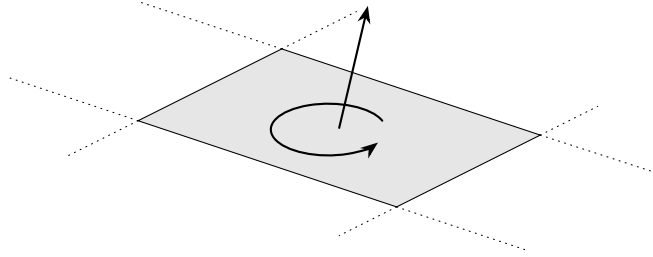
\begin{figure}
	\centering
	\begin{tikzpicture}
		[every node/.style={draw, circle, fill=black, inner sep=0.1pt, minimum size=0.5mm},
		baseline, scale=1.9, transform shape, line join=round, line cap=round,>=Stealth]
		
		\filldraw[fill=gray!20]
		(0,0.2) -- (1,0.7) -- (2.8,0.1) -- (1.8,-0.4) -- cycle;
		
		\draw [thick, ->] (1.4,0.15) -- (1.6, 1) {};
		\draw[thick,->](1.7,0.2)
		arc[
		start angle=20,
		end angle=330,
		x radius=0.4,
		y radius=0.18,
		rotate=-18
		];
		
		\draw[dotted] (0,0.2) -- (-0.9, 0.5);
		\draw[dotted] (1,0.7) -- (0, 1.03);
		\draw[dotted] (0,0.2) -- (-0.5, -0.05);
		\draw[dotted] (1,0.7) -- (1.52, 0.96);
		\draw[dotted] (1.8,-0.4) -- (2.7, -0.7);
		\draw[dotted] (2.8,0.1) -- (3.6, -0.16);
		\draw[dotted] (2.8, 0.1) -- (3.2, 0.3);
		\draw[dotted] (1.8,-0.4) -- (1.4, -0.6);	
	\end{tikzpicture}
	\caption{ A face of an orientable cellulated surface is oriented following the right hand rule along the outward normal vector of the face. Here the upward arrow denotes the outward normal vector of the face and the curved arrow depicts the orientation of the face.} \label{fig1}
\end{figure}

 Now, we remove a face from the cellulated surface $K$. We call this a \emph{``punctured''} cellulated surface and denote it as $K'$. Next we define a Pfaffian orientation on $K'$.
 \begin{defin}
 	Let $K'$ be a punctured, orientable, connected, cellulated surface where each edge is assigned an orientation (or a direction), represented by an arrow on the edge. Such an assignment of orientation on the edges of $K'$ is said to be a \emph{Pfaffian orientation} on $K'$ if for every face $f$ of $K'$, the number of edges $e$ of $f$, whose orientations are opposite to the induced orientation of $f$ on $e$, is odd.
 \end{defin}
 
 A Pfaffian orientation on a planar graph can be equivalently seen as a Pfaffian orienation on a ``punctured" $2$-sphere, embedded in $\mathbb{R}^3$.

 In Section~\ref{sec1}, we count the number of Pfaffian orientations on $K'$ and arrive at the following result.
 For any cellulated surface $K$, we denote the set of vertices of $K$ as $\mathcal{V}(K)$, the set of edges of $K$ as $\mathcal{E}(K)$ and the set of faces of $K$ as $\mathcal{F}(K)$.
 \begin{thm}\label{main}
 	Let $K$ be an orientable, closed, connected, cellulated surface of genus $g$, such that $|\mathcal{V}(K)|=v$ and $K'$ be the punctured, cellulated surface obtained by removing a face from $K$. Then the number of Pfaffian orientations on $K'$, denoted by $\mu$, is given by, $$ \mu = 2^{v-1+ 2g}.$$
 	
 \end{thm}
 As a corollary, we obtain the following result.
 \begin{cor}\label{cor}
 	The number of Pfaffian orientations on a punctured cellulated $2$-sphere embedded in $\mathbb{R}^3$, is given by, $$ \mu = 2^{v-1},$$ where $v$ is the number of vertices of the cellulated $2$-sphere.
 \end{cor}
 
We note here that the number of Pfaffian orientations on a punctured, orientable, cellulated surface, does not depend on the face being removed from the surface. Towards the end of Section~\ref{sec1}, we prove the existence of a Pfaffian orientation on a punctured, orientable, cellulated surface by providing an explicit construction.
 
	\section{Preliminaries}\label{prelim}
	Let $K$ be an orientable, connected, closed, cellulated surface. If an edge $e$ is contained in a face $f$, we denote it as $e \subseteq f$. The following two simple properties of such a surface will prove to be useful, subsequently.
	\begin{enumerate}[label=(\alph*)]
		\item Each edge of $K$ is contained in exactly two faces of $K$.
		\item For any two faces $f_0$ and $f_n$ of $K$, there exists a sequence of faces and edges of the following form.
		$$ f_0, e_0, f_1, \dots , e_{n-1}, f_n,$$
		where $f_i \in \mathcal{F}(K)$, $e_i \in \mathcal{E}(K)$, for each $i$ and $f_i \supseteq e_i \subseteq f_{i+1}$ for each $i \in \{0, \dots ,(n-1)\}$.
	\end{enumerate}

For a punctured, cellulated surface, an edge which is shared by exactly two faces of the surface, is termed as an \emph{internal edge}, while an edge which is contained in exactly one face of the surface, is termed as an \emph{external edge}.

	\begin{defin}{(Incidence Number)}
	Let $K$ be an orientable, cellulated surface with a given orientation on its cells. Let $e \in \mathcal{E}(K)$, $f \in \mathcal{F}(K)$. Then the \emph{incidence number} of $e$ with respect to $f$ is defined as,
	
	\begin{equation*}
		\langle f,e \rangle = \begin{cases}
			0,  \text{ 	if } e \subsetneq f,\\
			1,  \text{ if the orientation of $e$ is the induced orientation of $f$ on $e$},\\
			-1,  \text{ otherwise}.
		\end{cases}
	\end{equation*} 
	\end{defin}

 \begin{defin}
	The \emph{incidence matrix} of a cellulated surface $K$  is defined as a matrix where the columns correspond to all the faces of $K$, ordered as $f_1, f_2, \dots$, the rows correspond to all the edges of $K$, ordered as $e_1, e_2, \dots $, and the $(i,j)^{th}$ entry is given by the incidence number of the $i^{th}$ edge with respect to the $j^{th}$ face, i.e., $$ a_{ij}= \langle f_j, e_i\rangle.$$
We denote this matrix as $[\partial(K)]$.	
\end{defin}

	Now let us consider a punctured orientable surface $K'$ and consider the incidence matrix of $K'$, $[\partial(K')]$. We denote the $i^{th}$ column vector of $[\partial(K')]$ as $\bar{f}_i$ and the $i^{th}$ row vector of $[\partial(K')]$ as $\bar{e}_i$. Let $|\mathcal{F}(K)|=p$. The following result regarding the rank of $[\partial(K')]$, will prove to be crucial subsequently.
	\begin{lem}\label{li}
		The columns of $[\partial(K')]$ are linearly independent, i.e., the rank of $[\partial(K')]$ is $(p-1)$.
	\end{lem}
		\begin{proof}
		Let $\mathcal{F}(K)=\{f_1, \dots ,f_{p-1}, f_p\}$ and $\mathcal{F}(K')=\mathcal{F}(K) \setminus \{f_p\}$. Let $ \bar{f}_1, \bar{f}_2, \dots , \bar{f}_{p-1}$ be the respective column vectors of $[\partial(K')]$.
		
		Let,
		\begin{equation}
			\sum_{i=1}^{p-1}c_i\bar{f}_i=0, \tag{i}
		\end{equation}
		where $c_i \in \mathbb{Z}$ for each $i \in \{i, \dots ,(p-1)\}$. Without loss of generality, let $c_1 \ne 0$. As we noted earlier, there exists a sequence of faces and edges of $K$, $$ f_1, e_{i_1}, f_{i_{1}}, e_{i_{2}}, \dots ,f_{i_{t-1}}, e_{i_t}, f_p$$ from $f_1$ to $f_p$, where $f_{i_j} \in \mathcal{F}(K)$ and $e_{i_j} \in \mathcal{E}(K)$ for each $j \in [t]$. Without loss of generality we can assume that they are distinct. Since each $e_{i_j}$ is shared by exactly two faces of $K$, which are precisely the cells preceding and following $e_{i_j}$ in the sequence, therefore, in order to cancel the contribution of $e_{i_1}$ in (i), we must have $c_{i_1} \ne 0$. Again, to cancel the contribution of $e_{i_2}$ in (i), we must have $c_{i_2} \ne 0$. In this way we continue, till we have $c_{i_{t-1}} \ne 0$. As a result, $e_{i_t}$ now has a non-trivial contribution in the sum, which cannot be cancelled, $e_{i_t}$ being an external edge of $K'$. This is a contradiction. Therefore $c_1 = 0$. In general, $c_i = 0$ for each $i \in [p-1]$. Thus the column vectors $\bar{f}_1, \dots ,\bar{f}_{p-1}$ are linearly independent and hence the rank of $[\partial(K')]$ is $(p-1)$.
	\end{proof}
	It follows from the above result that the cardinality of any basis of the row space of $[\partial(K')]$ is $(p-1)$, which is same as $|\mathcal{F}(K')|$. Let us consider a basis of the row space of $[\partial(K')]$ and denote the set of edges corresponding to this basis, as $R$. Now we construct a directed graph as follows.\\

	\noindent \textbf{Construction of a directed bipartite graph:} Let $G(K')=(V,E)$ be a directed bipartite graph, where the vertex set $V$ is partitioned as $V= R \sqcup \mathcal{F}(K')$ and the edge set $E := \{(f,e) \mid f \in \mathcal{F}(K'), e \in R, e \subseteq f\}$. We will refer to $G(K')$ as simply $G$, when the surface is clear from the context. As we observed previously, $|R|=|\mathcal{F}(K')|$, which makes $G$ a balanced bipartite directed graph. Let $\mathcal{M} = \{M \mid M \text{ is a perfect matching in } G\}$. We will prove at the end of Section~\ref{sec1} that $\mathcal{M} \ne \emptyset$. Next, let $M \in \mathcal{M}$. For each directed edge $(f,e) \in M$, we reverse the direction of this edge in $G$. We denote this modified bipartite directed graph as $G_{M}$. We call a perfect matching $M$ in $G$ \emph{acyclic} if $G_{M}$ is acyclic. Otherwise, $M$ is called \emph{cyclic}.
	\subsection{Some definitions and results of Graph theory}
	Let $G=(V,E)$ be a directed graph, where $V$ and $E$ denote the set of vertices and the set of directed edges of $G$ respectively. We also denote the set of vertices of $G$ as $V(G)$ and the set of directed edges of $G$ as $E(G)$. Throughout the paper, we will refer to directed edges of a directed graph as simply \emph{edges} and directed cycles as simply \emph{cycles}. For any $u \in V$, let $N_u^{-}=\{v \in V \mid (u,v) \in E \}$ and  $N_u^{+}= \{v \in V \mid (v,u) \in E\}$. Then, the \emph{out-degree} of $u$ in $G$ is defined as $|N_u^{-}|$ and denoted as $\outdeg_G(u)$. The \emph{in-degree} of $u$ in $G$ is defined as $|N_u^{+}|$ and denoted as $\indeg_G(u)$. The elements of $N_u^{-}$ are said to be the \emph{out-neighbours} of $u$ while the elements of $N_u^{+}$ are said to be the \emph{in-neighbours} of $u$. If $G$ has no directed cycles, then we call $G$ a \emph{directed acyclic graph}. The following result is of immense importance, which we will use in the proof of \autoref{exis} in Section~\ref{sec1}.
	
	\begin{prop}\label{dag}
		A directed acyclic graph has atleast one vertex of out-degree $0$. 
	\end{prop}
We refer to \cite{West}, for more information on the basics of Graph Theory.
	\section{Enumerating Pfaffian orientations on a punctured cellulated surface}\label{sec1}
	
	Let $K'$ be a punctured cellulated orientable surface. In this section, we construct a system of equations over the field $\mathbb{Z}_2$, for such a cellulated surface, the solutions of which correspond to precisely the Pfaffian orientations on $K'$. Next, we determine the dimension of the solution space, which gives us the number of solutions of this system.
	
	\subsection{Construction of a system of equations}
	 Let $f$ be a face of $K'$. Then for each edge $e$ of $f$, we set a variable $x(f,e)$ which takes values as follows.
	
	$$x(f,e)= \begin{cases}
		1 & \text{if }\langle f,e\rangle =-1, \\
		0 & \text{otherwise.}\\
	\end{cases} $$
	
	Let $K'$ be equipped with an orientation on its edges. If this orientation is Pfaffian, then we observe that the variables with respect to $K'$ as mentioned before, should satisfy the following equations.
	For $f \in \mathcal{F}(K')$ with edges $e_1, e_2, \dots , e_k$, we have,
	 
	\begin{equation}\label{eqn1}
		x(f,e_1) + x(f,e_2) + \dots + x(f,e_k) =1
	\end{equation} 
	For each internal edge $e \in \mathcal{E}(K')$ such that $e$ is shared by the faces $f$ and $g$, we have,
	\begin{equation} \label{eqn2}
		x(f,e) + x(g,e) =1
	\end{equation}
	We note here that this system of equations are taken over $\mathbb{Z}_2$.
	
	We call \ref{eqn1} the set of \emph{face equations} and \ref{eqn2} the set of \textit{edge equations}. The face equations impose the condition of parity of the number of good cells with respect to each face of $K'$, while the edge equations ensure that the orientation on the edges of $K'$ are coherent, which is necessary for orientability of the surface. Hence, we further observe that each solution of the system of equations formed by \ref{eqn1} and \ref{eqn2} corresponds to a Pfaffian orientation on $K'$.
	
	Now, the system of equation can be written as $AX= \mathbb{1}$ where $A$ is the coefficient matrix, $X$ is the variable matrix. Let the dimension of $A$ be $r \times s$. Then $\mathbb{1}$ denotes a $r \times 1$ matrix consisting of all $1$-s. We call $A$ the \textit{Pfaffian matrix} of $K'$. We demonstrate this matrix with the help of an illustrative example.
			\begin{exmp}
		Let us consider a punctured, cellulated torus as depicted in \autoref{torus}. We name it as $\mathbb{T}$.
		
		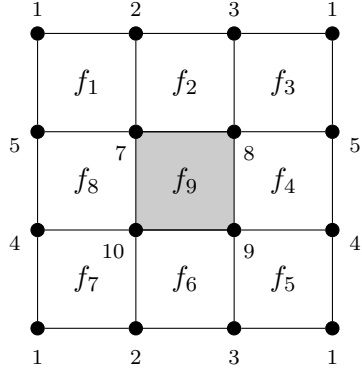
\begin{figure}
			\centering
			\begin{tikzpicture}[vertex/.style={draw, circle, fill=black, inner sep=0.05pt, minimum size=1.75mm},baseline, scale=1.3]
				
				\node[vertex] [label={[yshift=-0.7cm]\scriptsize{$1$}}] (a) at (0,0) {};
				\filldraw[fill=gray!40!white] (1,1) rectangle (2,2);
				\node[vertex] [label={[xshift=-0.3cm, yshift=-0.5cm]\scriptsize{$4$}}] (b) at (0,1) {};
				\node[vertex] [label={[xshift=-0.3cm, yshift=-0.5cm]\scriptsize{$5$}}] (c) at (0,2) {};
				\node[vertex] [label={[yshift=0 cm]\scriptsize{$1$}}] (d) at (0,3) {};
				\node[vertex] [label={[yshift=-0.7cm]\scriptsize{$2$}}] (e) at (1,0) {};
				\node[vertex] [label={[yshift=-0.7cm]\scriptsize{$3$}}] (f) at (2,0) {};
				\node[vertex] [label={[yshift=-0.7cm]\scriptsize{$1$}}] (g) at (3,0) {};
				\node[vertex] [label={[yshift=0.005cm]\scriptsize{$2$}}](h) at (1,3) {};
				\node[vertex] [label={[yshift=0.005cm]\scriptsize{$3$}}](i) at (2,3) {};
				\node[vertex] [label={[yshift=0cm]\scriptsize{$1$}}] (j) at (3,3) {};
				\node[vertex] [label={[xshift=0.3cm, yshift=-0.5cm]\scriptsize{$5$}}] (m) at (3,2) {};
				\node[vertex] [label={[xshift=0.3cm, yshift=-0.5cm]\scriptsize{$4$}}] (k) at (3,1) {};
				\node[vertex] [label={[xshift=-0.3cm, yshift=-0.6cm]\scriptsize{$10$}}](l) at (1,1) {};
				\node[vertex] [label={[xshift=0.2cm, yshift=-0.6cm]\scriptsize{$9$}}](n) at (2,1) {};
				\node[vertex] [label={[xshift=-0.2cm, yshift=-0.6cm]\scriptsize{$7$}}](o) at (1,2) {};
				\node[vertex] [label={[xshift=0.2cm, yshift=-0.6cm]\scriptsize{$8$}}](p) at (2,2) {};
				\node at (0.5, 0.5) {$f_7$};
				\node at (1.5, 0.5) {$f_6$};
				\node at (2.5, 0.5) {$f_5$};
				\node at (2.5, 1.5) {$f_4$};
				\node at (2.5, 2.5) {$f_3$};
				\node at (1.5, 2.5) {$f_2$};
				\node at (0.5, 2.5) {$f_1$};
				\node at (0.5, 1.5) {$f_8$};
				\node at (1.5, 1.5) {$f_9$};
				
				\draw (a)--(b);
				\draw (b)--(c);
				\draw (c)--(d);
				\draw (a)--(e);
				\draw (e)--(f);
				\draw (f)--(g);
				\draw (b)--(l);
				\draw (l)--(n);
				\draw (n)--(k);
				\draw (c)--(o);
				\draw (o)--(p);
				\draw (p)--(m);
				\draw (d)--(h);
				\draw (h)--(i);
				\draw (i)--(j);
				\draw (g)--(k);
				\draw (k)--(m);
				\draw (m)--(j);
				\draw (e)--(l);
				\draw (f)--(n);
				\draw (o)--(h);
				\draw (p)--(i);

			\end{tikzpicture}
			\caption{A punctured cellulated torus where the shaded cell $f_9$ represents the punctured face.} \label{torus}
		\end{figure}

		Next, we form the face equations and edge equations of $\mathbb{T}$ as explained above, and obtain a system of equations $AX = \mathbb{1}$, where the Pfaffian matrix is of dimension $14*32$ and  given by,
		
		$$ A =
		\setcounter{MaxMatrixCols}{30} 
		\tiny
		\begin{bmatrix}
			(f_{1},12) & (f_{1},27) & (f_{1},57) & (f_{1},15) & (f_{2},27) & (f_{2},23)& (f_{2},38) & (f_{2},78) & & \cdots & & (f_{8},57) & (f_{8},(7,10)) & (f_{8},(4,10)) & (f_{8},45)   \\ 
			1 & 1 & 1 & 1 & 0 & 0 & 0 & 0 & & \cdots & & 0 & 0 & 0 & 0  \\ 
			0 & 0 & 0 & 0 & 1 & 1 & 1 & 1 & & \cdots & &0 & 0 & 0 & 0 \\
			\vdots &  &  &  &  &  &  &  & & & &  &  &  &   \\
			0 & 0 & \cdots & & & & & & & & \cdots & 1 & 1 & 1 & 1\\
			1 & 0 & 0 & 0 & \cdots &  &  &  &\cdots& 1 & \cdots & 0 & 0 & 0 & 0\\
			0 & 1 & 0 & 0 & 1 & 0 & 0& &\cdots& & & 0 & 0 & 0 & 0 \\
			0 & 0 & 1 & 0 & \cdots & & & & & & & 1 & 0 & 0 & 0 \\
			0 & 0 & 0 & 1 & \cdots & & & & 1 & \cdots& & 0 & 0 & 0& 0 \\
			\vdots & & & & & & & & & & & 0 & 0 & 0 & 1 \\
			
		\end{bmatrix}$$  and $$ X = \begin{bmatrix}
			x(f_{1},12)\\
			x(f_{1},27)\\
			x(f_{1},57)\\
			x(f_{1},15)\\
			\vdots\\
			x(f_{8},45)\\
		\end{bmatrix}. $$

	\end{exmp}
	Hence we arrive at the following result.
	
	\begin{lem}\label{count}
		Let $K'$ be a punctured orientable cellulated surface and $A$ be its Pfaffian matrix. Then each solution of the system of equation $AX= \mathbb{1}$ corresponds to a Pfaffian orientation on $K'$ and vice versa. Hence, the number of Pfaffian orientations on $K'$ is given by the number of solutions of this system of equations.
	\end{lem}
	So, in order to find out the number of Pfaffian orientations on $K'$, it suffices now to find the number of solutions of this system of equations, which leads us to the next subsection.
	
	\subsection{Enumerating Pfaffian orientations on a punctured, cellulated, orientable surface}
	Let $A$ be the Pfaffian matrix of $K'$. In this subsection, we first show that $A$ can be reduced to an equivalent matrix $A'$. Next, we find out $\nullity(A')$ and hence, determine the number of solutions of the system $AX=\mathbb{1}$.
	The following theorem states how the Pfaffian matrix of a punctured cellulated surface can be reduced to another equivalent matrix. For any matrix $B$, we denote the transpose of $B$ as $B^{T}$.
	
	\begin{prop}\label{pfaf}
		Let $K'$ be a punctured orientable cellulated surface with $|\mathcal{F}(K')|=(p-1)$ and $|\mathcal{E}(K')|=d$. Let $A$ be the Pfaffian matrix of $K'$ of order $r \times s$ and $AX=\mathbb{1}$  be the corresponding system of equations. Then $A$ can be reduced to another equivalent matrix $A'$ where,
		$$ A' = \begin{bmatrix}
			[\partial(K')]^{T} & O\\
			O &  I_{s-d}\\
		\end{bmatrix}, $$
		and $I_{s-d}$ is the identity matrix of order $(s-d)$, by a sequence of column and row operations.
	\end{prop}
		\begin{proof}
		
		First, we order the faces of $K'$ according to the order of the face equations in the system of equations $AX= \mathbb{1}$. Let this order be $f_1, \dots ,f_{p-1}$. Similarly, we order the internal edges of $K'$ according to the order of the edge equations, as $e_{j_1}, \dots ,e_{j_l}$. We recall that the first $(p-1)$ rows of $A$ correspond to the face equations while the remaining rows correspond to the edge equations. 
		
		Let $C_{ij}$ denote the column of coefficients of the variable $x(f_i,e_j)$. Thus, each column $C_{ij}$ of $A$ is of the following form. If $e_j$ is an internal edge, then let $e_j=e_{j_t}$ for some $t \in [l]$. In this case, $C_{ij}$ is of the form,
		\[
		\begin{tikzpicture}[baseline=(m.center), scale=0.5]
			\matrix (m) [matrix of math nodes,left delimiter={[},right delimiter={]}] {
				0 \\
				\vdots \\
				0 \\
				1 \\
				0 \\
				\vdots \\
				0 \\
				1 \\
				0 \\
				\vdots \\
				0 \\
			};
			
			\draw[->] (m-4-1.east) -- ++(1,0) node[right] {$i^{th}$ position};
			
			\draw[->] (m-8-1.east) -- ++(1,0) node[right] {$(p-1+t)^{th}$ position};
			
		\end{tikzpicture}.
		\]
		However, if $e_j$ is a external edge, then $C_{ij}$ is of the following form,
		\[
		\begin{tikzpicture}[baseline=(m.center), scale=0.5]
			\matrix (m) [matrix of math nodes,left delimiter={[},right delimiter={]}] {
				0 \\
				\vdots \\
				0\\
				1 \\
				0 \\
				\vdots \\
				0 \\
			};
			
			\draw[->] (m-4-1.east) -- ++(1,0) node[right] {$i^{th}$ position};
			
		\end{tikzpicture}.
		\]
		Now we perform a sequence of column operations. First we take $i=1$. For each $C_{1j}$, if $e_j$ is an internal edge and shared by another face $f_k$, $k >1$, then we perform the following column operation on $C_{1j}$, $C_{1j}'= C_{1j} + C_{kj}$. Now the $i^{th}$ and $k^{th}$ entry of $C_{1j}'$ are $1$ while the remaining entries are $0$. If, $e_j$ is a external edge, then we keep $C_{1j}$ unaltered.
		
		In general, for $i > 1$, we have the following scheme. Let $e_j$ be an internal edge and shared by another face $f_k$. If $k > i$, then we perform the column operation $C_{ij}'=C_{ij}+C_{kj}$, whereby, the $i^{th}$ and $k^{th}$ entry of $C_{ij}'$ become $1$ and the remaining entries become $0$. If $k < i$, then we move $C_{ij}$ to the terminal position in the matrix. If $e_j$ is a external edge, we keep $C_{ij}$ unaltered. We carry out this procedure for each $i \in \{2, \dots ,(p-1)\}$. Let us name this column-transformed matrix as $B$.
		
		Now, we observe that for each $C_{ir}, C_{ks}$ in the first $d$-columns of $B$, $r \ne s$. Further, each $C_{ij}$ in the first $d$ columns is of the following form. If $e_j$ is shared by $f_i$ and $f_k$, the the $i^{th}$ and $k^{th}$ entries of $C_{ij}$ are $1$ while the remaining are $0$. If $e_j$ is a external edge, then the $i^{th}$ entry of $C_{ij}$ is $1$ while the remaining entries are $0$.  Thus the submatrix formed by the first $d$ columns of $B$, is of the following form,
		
		\[\begin{bmatrix}
			[\partial(K')]^{T}\\
			O \\
		\end{bmatrix}
		\]
		where, we note, the dimension of the null matrix is $(r-p+1) \times d$.
		We now observe that for each $C_{ij}$ in the remaining $(s-d)$ columns of $A$, $e_j$ is an internal edge. Further, for each $C_{ij}, C_{kl}$ in these $(s-d)$ columns, $j \ne l$. So, we observe that $(s-d)$ corresponds to precisely the number of internal edges of $K'$, which is again same as the number of edge equations, $(r-p+1)$. We recall that the internal edges are ordered as $e_{j_1}, \dots , e_{j_l} $. We now shift these $(s-d)$ columns amongst themselves by successive column operations, so that they are ordered as $C_{i_1j_1}, C_{i_2j_2}, \dots , C_{i_lj_l}$. Since as noted before, $(s-d)=(r-p+1)$, therefore according to the property of these columns as stated at the beginning of the proof, we observe that the submatrix formed by the terminal $(s-d)$ columns after the aforementioned column operations on the terminal $(s-d)$ columns of $B$, is of the following form.
		
		\[\begin{bmatrix}
			D\\
			I_{s-d}\\
		\end{bmatrix}
		\]
		where $D$ is a binary matrix. Now, we can apply suitable row operations on $D$  using $I_{s-d}$, so as to reduce $D$ to a null matrix. Therefore, $B$ is reduced to $A'$, which is of the following form.
		$$ A' = \begin{bmatrix}
			[\partial(K')]^{T} & O\\
			O &  I_{s-d}\\
		\end{bmatrix} $$
		
		This completes the proof.
	\end{proof}
	
	Now, we determine $\nullity(A')$ using the following lemma.
	
			\begin{lem}\label{nullity}
		Let $|\mathcal{F}(K')|=(p-1)$, $|\mathcal{E}(K')|=d$ and $A'$ be the matrix obtained from the Pfaffian matrix of $K'$, as asserted in \autoref{pfaf}. Then $\nullity(A')= d-p+1$.
	\end{lem}
	\begin{proof}
	 Suppose the dimension of $A$ is $r \times s$. We recall that $A'$ is of the form,
		
		$$ A' = \begin{bmatrix}
			[\partial(K')]^{T} & O\\
			O &  I_{s-d}\\
		\end{bmatrix} $$
		Let,
		$$ B= \begin{bmatrix}
			[\partial(K')]^{T}\\
			O \\
		\end{bmatrix} $$
		It follows that the set of all the linearly independent columns of $A'$ is comprised of precisely, the terminal $(s-d)$ columns of $A'$ and the linearly independent columns of $B$, since the linear independence of the former is not perturbed by the latter and vice versa. From \autoref{li}, we know that rank of $B$ is $(p-1)$. Therefore, \begin{equation*}
		\begin{aligned}
			\nullity(A')&= s - ((s-d) +(p-1))\\ &= d- p+1
		\end{aligned}
		\end{equation*}
	\end{proof}
	
	Next, we prove the existence of a Pfaffian orientation on a punctured, cellulated, orientable surface and hence determine the number of Pfaffian orientations on such a surface using \autoref{nullity}.
	 
	We mention here, that Kastelyn has proved the existence of Pfaffian orientations in planar graphs. However, his proof will not work for Pfaffian orientations on punctured cellulated surfaces in general, since he used planarity to prove his result. We will use the notion of the directed bipartite graph and the associated acyclic perfect matchings on it, as described in Section~\ref{prelim}, to provide an explicit construction of a Pfaffian orientation on any punctured, orientable, cellulated surface. As before, let $K'$ be a punctured, orientable, cellulated surface, with $|\mathcal{F}(K')|=(p-1)$.
	
	First we recall our construction from Section~\ref{prelim}. We consider a basis of the row space of $[\partial(K')]$ and name the set of edges corresponding to this basis as $R$. We construct a directed balanced bipartite graph $G=(V,E)$, where $V= R \sqcup \mathcal{F}(K')$, $E := \{(f,e) \mid f \in \mathcal{F}(K'), e \in R, e \subseteq f\}$ and both $R$ and $\mathcal{F}$ are given an ordering on their vertices. The set of all possible perfect matchings on $G$ is denoted as $\mathcal{M}$. We further recall that each $M \in \mathcal{M}$ gives rise to a modified directed bipartite graph $G_M$. If $G_M$ is acyclic, we call $M$ an acyclic matching. Otherwise, $M$ is called cyclic.   
	
	Now, for each $M \in \mathcal{M}$ in $G$, we associate a permutation $\pi_M$ as follows. If $(f_j,e_i) \in M$, where $f_j$ is the $j^{th}$ vertex in $\mathcal{F}(K')$ and $e_i$ is the $i^{th}$ vertex in $R$, then $\pi_M(i)=j$. We define a sign on $M$ as $\Lambda(M)= \sgn(\pi_M)\prod_{(f,e) \in M}\langle f, e \rangle$. The following lemma provides a crucial information on the sum of the signs of the perfect matchings in $G$, which establishes the existence of a perfect matching in $G$. 
	
	\begin{lem}\label{sign}
		$\sum_{M \in \mathcal{M}}\Lambda(M) \ne 0$.
	\end{lem}
	
	\begin{proof}
			Let us consider $[\partial(K')]$. Suppose, $R$ is the set of edges corresponding to a basis of the row space of $[\partial(K')]$. Let us denote the submatrix formed by all the columns of $[\partial(K')]$ and the rows of this basis as $\partial_R$. Without loss of generality, let us assume that the rows and columns of $\partial_R$ are arranged in the order of the vertices of $R$ and $\mathcal{F}(K')$ in $G(K')$ respectively. Therefore from \autoref{li}, we deduce that $\det(\partial_R(K')) \ne 0$. We observe here that $\sum_{M \in \mathcal{M}}\Lambda(M)= \det(\partial_R(K'))$. Indeed, from the definition of determinant, we have,
		
		\begin{equation*}
			\begin{aligned}
				\det(\partial_R) &= \sum_{\sigma \in S_n} \sgn(\sigma) \prod_{i=1}^{|R|}a_{i\sigma(i)}\\ & = \sum_{\sigma \in S_n} \sgn(\sigma) \prod_{i=1}^{|R|} \langle f_{\sigma(i)}, e_i \rangle.\\
			\end{aligned}
		\end{equation*}
		We observe that the summands in this expression survive only when $e_i \subseteq f_{\sigma(i)}$ for each $i \in \{1, \dots , |R|\}$. Equivalently, a summand corresponding to some $\sigma \in S_n$ in the above expression is nontrivial iff $\{(f_{\sigma(i)},e_i) \mid i=1, \dots , |R|\}$ forms a perfect matching in $G(K')$. The permutation associated with this perfect matching is precisely $\sigma$. Thus the expression reduces to, 
		\begin{equation*}
			\begin{aligned}
				\det(\partial_R) &= \sum_{M \in \mathcal{M}} \sgn(\pi_M) \prod_{(f,e) \in M}\langle f, e \rangle\\ & = \sum_{M \in \mathcal{M}} \Lambda(M).
			\end{aligned}
		\end{equation*}
		Hence, the result follows.
	\end{proof}
	
	Next, we construct a sign-inversing involution on the set of cyclic matchings of $G$, thereby establishing the existence of an acyclic perfect matching in $G$. This acyclic matching will help us to provide a Pfaffian orientation on $K'$.
	
	From here on, an ordered pair $(e,f)$ will denote a directed edge from $e \in R$ to $f \in \mathcal{F}$ and $(f,e)$  will denote a directed edge from $f \in \mathcal{F}$  to  $e \in R$.
	
	We mention here that the proofs of \autoref{inv} and \autoref{acyc} have been adapted from similar results in \cite{Mukherjee}. Before delving into these results, we note a couple of observations which will be useful subsequently. These observations follow directly from the construction of $G_M$ and the properties of cellulated surfaces, as stated at the beginning of Section~\ref{prelim}.
	\begin{obs}\label{deg}
		Let $M$ be a perfect matching in $G$. Then, 
		\begin{enumerate}[label=(\alph*)]
			\item For each $v \in R$, $\outdeg_{G_M}(v)= \indeg_{G_M}(v)=1$.
			\item For each $v \in \mathcal{F}$, $\indeg_{G_M}(v)=1$.
		\end{enumerate}
	\end{obs} 
	\begin{lem}\label{inv}
		$$\sum_{\substack{M \in \mathcal{M} \\ M \text{ is cyclic}}}\Lambda(M)=0.$$
	\end{lem}
	
	\begin{proof}
		Let $M \in \mathcal{M}$ such that $M$ is cyclic. We will construct another cyclic matching $\widetilde{M}$ corresponding to $M$ and show that this correspondence is in fact, a sign-inversing involution. 
		
		Let $C_1, C_2, \dots ,C_k$ be the directed cycles in $G_M$. We claim that the cycles in $G_M$ are vertex-disjoint, i.e., $V(C_i) \cap V(C_j) = \emptyset$, for each $i, j \in [k], i \ne j$. This is because, if $V(C_i) \cap V(C_j) \ne \emptyset$ for some $i, j \in [k], i \ne j$, then, there must exist a vertex $v$ in $C_i$ with in-degree $2$, which we note from \autoref{deg}, is not possible. Now, the vertex-disjointness of the cycles in $G_M$ leads us to the observation that for any cycle $C$ in $G_M$, the set of vertices $V(C) \cap \mathcal{F}$ is unique. So, for each cycle $C$ in $G_M$, we arrange the vertices of $V(C) \cap \mathcal{F}$ in the lexicographic order and assign this tuple of ordered vertices to $C$. Let us denote this tuple corresponding to a cycle $C$ in $G_M$ as $t_C(G_M)$. Now, we consider $\{t_C \mid C \text{ is a cycle in }G_M\}$ and arrange the elements of this set in lexicographic order. Let $C'$ be the cycle corresponding to the first tuple in this order. 
		Now, we simply reverse the direction of the edges of $C'$, keeping all other edges unchanged, and rename the resultant graph as $G'$. We construct $\widetilde{M}$ as ,
		$$ \widetilde{M} = \{(f,e) \in E(G) \mid (e,f) \in E(G')\}.$$
It can be verified using \autoref{deg} that $\widetilde{M}$ is indeed a matching. Since, every cycle in a bipartite graph is of even length, therefore, from the construction of $G'$, we observe that $|M|=|\widetilde{M}|$, which makes $\widetilde{M}$ perfect. Indeed, $G_{\widetilde{M}}=G'$.
Again, from the vertex-disjointness of the cycles, we observe that the vertex set of each cycle in $G_{\widetilde{M}}$ is preserved under this operation. Let the cycle in $G_{\widetilde{M}}$, which has been obtained by reversing $C'$, be denoted as $C''$. As a result, $t_{C'}(G_M)=t_{C''}(G_{\widetilde{M}})$ and hence $t_{C''}$ comes first in the lexicographic ordering of the tuples of all the cycles in $G_{\widetilde{M}}$. Therefore, reversing the direction of the edges of $C''$ in $G_{\widetilde{M}}$, returns us $G_M$, which corresponds to the original matching $M$. So, this correspondence is an involution. 

Now, we show, $\Lambda(M) = - \Lambda(\widetilde{M})$. Let $S=\{(f,e) \in M \mid (e,f) \in E(C')\}$ and $T= \{(f,e) \in \widetilde{M} \mid (e,f) \in E(C'')\} $. A few observations are worthy of mention here. First, we observe that $M \setminus S= \widetilde{M} \setminus T$, which follows from the construction of $\widetilde{M}$. Next, we observe here that for every $(f,e) \in S$, there exists $ f' \in \mathcal{F}$, $f' \ne f$, such that $(f',e) \in T$ (follows from the construction of $\widetilde{M}$) and $\langle f,e \rangle = - \langle f', e\rangle$ (follows from the orientability of the surface). Similarly, for every $(f,e) \in T$, there exists $ f' \in \mathcal{F}$, $f' \ne f$, such that $(f',e) \in S$ and $\langle f,e \rangle = - \langle f', e\rangle$. Further, it can be noted that $\pi_{\widetilde{M}}= \rho_{|S|} \circ \pi_M$, where $\rho_{|S|}$ is a permutation of length $|S|$. Thus, $\sgn(\pi_{\widetilde{M}})= (-1)^{|S|+1} \sgn(\pi_M)$. So, we have,
		
		\begin{equation*}
			\begin{aligned}
				\Lambda(M) &= \sgn(\pi_M)\prod_{(f,e) \in M}\langle f, e \rangle\\ &= \sgn(\pi_M) \left(\prod_{(f,e) \in M \setminus S}\langle f, e \rangle\right) \left(\prod_{(f,e) \in S}\langle f, e \rangle\right)\\
				&= \sgn(\pi_M) (-1)^{|S|}\left( \prod_{(f,e) \in M \setminus S}\langle f, e \rangle\right)\left( \prod_{(f,e) \in T}\langle f, e \rangle\right)\\
				&= - \sgn(\pi_{\widetilde{M}})\left( \prod_{(f,e) \in \widetilde{M} \setminus T}\langle f, e \rangle \right)\left(\prod_{(f,e) \in T}\langle f, e \rangle\right)\\
				&= - \sgn(\pi_{\widetilde{M}})\prod_{(f,e) \in \widetilde{M}}\langle f, e \rangle\\
				&= -\Lambda(\widetilde{M}).
			\end{aligned}
		\end{equation*}
		
		Thus, the contribution of the sign of each cyclic perfect matching $M$ in the sum, gets cancelled by the sign of its corresponding cyclic perfect matching $\widetilde{M}$. This proves the result.
	\end{proof}
	
	\begin{cor}\label{acyc}
		For any punctured cellulated orientable surface $K'$, there exists an acyclic perfect matching in $G(K')$. 
	\end{cor}
	
	\begin{proof}
		This follows as a direct consequence of \autoref{sign} and \autoref{inv}.
	\end{proof}
	
	We introduce a term here which we will use in the next proof. Let $f \in \mathcal{F}(K')$, where $K'$ is a punctured, orientable, cellulated surface, with an orientation assigned on each of its edges. Then, an edge $e$ of $f$ is said to be a ``good" edge of $f$ if the orientation of $e$ is opposite to the induced orientation of $f$ on $e$. Thus, an assignment of orientations on the edges of $K'$ is Pfaffian, iff every face of $K'$ has odd number of good edges.
	\begin{thm}\label{exis}
		Any punctured cellulated orientable surface admits a Pfaffian orientation.
	\end{thm}
	\begin{proof}
		Our line of argument is aligned parallelly with Kastelyn's proof of existence of Pfaffian orientations for planar graphs.
		
		Let $K'$ be a punctured, cellulated, orientable surface and $G(K')$ be the directed bipartite graph associated with $K'$, as described previously. Therefore, from \autoref{acyc}, it follows that there exists an acyclic perfect matching $M$ in $G$. So, $G_M$ is acyclic. Thus, we recall from \autoref{dag} that there exists a vertex $f_1$ in $G_M$ with out-degree $0$. From \autoref{deg}, we infer that $f_1 \in \mathcal{F}$ and further, there exists a unique $e_1 \in R$ such that $(e_1, f_1) \in E(G_M)$. Now, if $f_1$ has odd number of good edges then we keep the orientation of $e_1$ unaltered. Otherwise, we reverse the orientation of $e_1$ so that $f_1$ has odd number of good edges. We now delete the vertices $e_1,f_1$ from $G_M$ and name the resultant graph $G_M'$. $G_M'$ is also acyclic and hence has a vertex $f_2$ of outdegree $0$. We note from \autoref{deg} that the deletion of $e_1, f_1$ does not affect the out-degree of the vertices in $R$ and the in-neighbours of the vertices in $\mathcal{F}$, in $G_M'$. So, again, $f_2 \in \mathcal{F}$ and there exists a unique $e_2 \in \mathcal{F}$ such that $(e_2, f_2) \in E(G_M')$. Now, as before, if $f_2$ has odd number of good edges, then we keep the orientation of $e_2$ unaltered, otherwise we reverse the orientation of $e_2$ so that $f_2$ has odd number of edges. A crucial observation here is that the reorientation of $e_2$ does not disturb the orientation of the edges in $f_1$ since $\outdeg_{G_M}(f_1)=0$, which makes it impossible for $e_2$ to be an out-neighbour and hence, an edge of $f_1$. So, now both $f_1$ and $f_2$ has odd number of good edges. Now, we delete $e_2, f_2$ from $G_M'$ and move on to the next vertex of outdegree $0$ in the resultant acyclic graph, which by the same argument, is a vertex of $\mathcal{F}$. We continue the same procedure till all the vertices of $G_M$ have been deleted.  At each step, we ensure that the faces of $K'$ have odd number of good edges by an appropriate reorientation of its unique in-neighbour in $G_M$. The resulting orientation on the edges of $K'$ is hence a Pfaffian orientation on $K'$.
	\end{proof}
	
	Now, we are ready to prove our main theorem.
	
	\begin{proof}[Proof of \autoref{main}]
		It follows from \autoref{count} that the number of Pfaffian orientations of $K'$ is given by the number of solutions of the system of equations $AX=\mathbb{1}$, where $A$ is the Pfaffian matrix of $K'$. Let $|\mathcal{F}(K)|=p$, $|\mathcal{E}(K)|=d$. Then, from \autoref{pfaf} and \autoref{nullity}, we deduce that $\nullity(A)=(d-p+1)$.We recall that the Euler characteristic of $K$ is given by $\chi(K)= (v - d + p) =(2 -2g)$, where, $|\mathcal{V}(K)|=v$. Therefore, we have, $(d-p+1)=(v-1+2g)$.  Hence, from \autoref{exis}, we conclude that the number of solutions of the system $AX=\mathbb{1}$ is given by $2^{v-1+2g}$.  Hence, the result follows.
	\end{proof}
	
	We obtain \autoref{cor} as an immediate result of \autoref{main}.
	
		\bibliographystyle{plain}
	\bibliography{ref.bib}
	\end{document}